\documentclass[reqno]{article}

\usepackage[colorinlistoftodos]{todonotes}
\usepackage[utf8]{inputenc}
\usepackage{amsmath,amsfonts,amssymb,amsthm}
\usepackage{mathtools}
\usepackage{enumitem}

\usepackage{hyperref}
\hypersetup{
    colorlinks=true,
    linkcolor=blue,
    citecolor=red,
    urlcolor=blue,
    pdfborder={0 0 0}
}

\usepackage{cleveref}
  \crefname{section}{Section}{Sections}
  \crefname{figure}{Figure}{Figures}
  \crefname{theorem}{Theorem}{Theorems}
  \crefname{lemma}{Lemma}{Lemmas}
  \crefname{proposition}{Proposition}{Propositions}  
  \crefname{corollary}{Corollary}{Corollaries}
  \crefname{definition}{Definition}{Definitions}
  \crefname{example}{Example}{Examples}
  \crefname{remark}{Remark}{Remarks}

\newtheorem{theorem}{Theorem}[section]
\newtheorem{lemma}[theorem]{Lemma}
\newtheorem{proposition}[theorem]{Proposition}
\newtheorem{corollary}[theorem]{Corollary}

\theoremstyle{definition}

\theoremstyle{remark}
\newtheorem{remark}[theorem]{Remark}

\DeclarePairedDelimiter\abs{\lvert}{\rvert}
\DeclarePairedDelimiter\norm{\lVert}{\rVert}

\newcommand*{\defeq}{\mathrel{\mathop:}=}
\newcommand*{\eqdef}{=\mathrel{\mathop:}}

\renewcommand*{\Re}{\operatorname{Re}}
\renewcommand*{\Im}{\operatorname{Im}}

\newcommand*{\diam}{\operatorname{diam}}
\newcommand*{\dist}{\operatorname{dist}}
\newcommand*{\hcap}{\operatorname{hcap}}

\renewcommand*{\fill}{\operatorname{fill}}

\newcommand*{\ex}{\mathbb E}

\newcommand*{\NN}{\mathbb N}

\newcommand*{\RR}{\mathbb R}
\newcommand*{\CC}{\mathbb C}
\newcommand*{\hatC}{\hat\CC}
\newcommand*{\HH}{\mathbb H}
\newcommand*{\barH}{\overline\HH}
\newcommand*{\DD}{\mathbb D}

\newcommand*{\slek}{SLE$_\kappa$}

\title{Topological characterisations of Loewner traces}
\author{Yizheng Yuan\footnote{TU Berlin, Germany. \texttt{yuan@math.tu-berlin.de}}}

\begin{document}

\maketitle

\begin{abstract}
The (chordal) Loewner differential equation encodes certain curves in the half-plane (aka traces) by continuous real-valued driving functions. Not all curves are traces; the latter can be defined via a geometric condition called the local growth property. In this paper we give two other equivalent conditions that characterise traces: 1. A continuous curve is a trace if and only if mapping out any initial segment preserves its continuity (which can be seen as an analogue of the domain Markov property of SLE). 2. The (not necessarily simple) traces are exactly the uniform limits of simple traces. Moreover, using methods by Lind, Marshall, Rohde (2010), we infer that uniform convergence of traces imply uniform convergence of their driving functions.
\end{abstract}

\section{Introduction and main results}

Loewner chains provide a way to encode certain curves in a planar domain by real-valued functions called driving functions or Loewner transforms. They had been originally introduced by K. Löwner (1923) as an approach to solve the Bieberbach conjecture, but have recently also been used by O. Schramm (2000) to construct Schramm-Loewner evolution (SLE) which is a random curve driven by a multiple of Brownian motion. The relation between the driving function and the corresponding curve (called trace) is quite involved. In particular, not all curves are traces, but only those that satisfy a geometric condition called the local growth property. (Conversely, not all driving functions do generate a trace either, and there is so far no known characterisation of such driving functions.)

Particularly nice Loewner traces are the so-called simple traces which do neither intersect themselves nor the boundary of the domain. But already SLE produce (for some parameters) examples of non-simple traces. Therefore there is motivation to study the space of (not necessarily simple) Loewner traces. In the following, we will consider chordal Loewner traces in the upper half-plane $\HH$. In \cite{TY20} the authors have shown that uniform limits of simple traces provide a (in general not simple) trace again, and they have raised the question whether the converse is true, i.e. whether any trace can be approximated by simple traces. (For \slek{} this has been known from \cite{LSW04,Tran15}.) We show in the present paper that this is indeed the case.

Another motivation for studying the space of Loewner traces is characterising the topological support of \slek{} (as a probability measure on the path space). In \cite{TY20} the authors have shown that the support of \slek{} is the closure of the set of simple traces. The result in the present paper implies that this is already the entire space of Loewner traces.

The main result of the present paper is the following characterisation of chordal Loewner traces. See \cref{sec:preliminaries} for definitions of the terminology.
\begin{theorem}\label{th:trace_characterisation}
Let $\gamma\colon {[0,\infty[} \to \barH$ be a continuous path with $\gamma(0) \in \RR$ such that the family of $K_t \defeq \fill(\gamma[0,t])$ is strictly increasing. Then the following are equivalent:
\begin{enumerate}[label=(\roman*)]
\item\label{it:lgp} The family $(K_t)_{t \ge 0}$ satisfies the local growth property.
\item\label{it:markov} For every $t \ge 0$, the path $\gamma_t(s) \defeq g_t(\gamma(s))$, $s \ge t$, is continuous.
\item\label{it:limit} There exists a sequence of simple paths $\gamma^n\colon {[0,\infty[} \to \barH$ with $\gamma^n(0) \in \RR$ and $\gamma^n(]0,\infty[) \subseteq \HH$ such that $\gamma^n \to \gamma$ locally uniformly.
\end{enumerate}
\end{theorem}

We point out that we identify $K_t$ by their intersection with $\HH$ (see \cref{sec:preliminaries}), for instance $\gamma \subseteq \RR$ are not counted as strictly increasing.

\begin{remark}
To be very precise, a boundary point $z \in \partial K_t$ can belong to several prime ends of $\HH \setminus K_t$, so the image $g_t(z)$ would not be unique. Therefore the precise formulation of \ref{it:markov} is that $\gamma_t$ can be chosen to be continuous (in case $\gamma(s) \in \partial K_t$ for some $s > t$).
\end{remark}

While all of the above properties seem natural, proving their equivalence requires some work. One should keep in mind that Loewner traces might have infinitely many self-intersections and be space-filling (e.g. \slek{} with $\kappa \ge 8$). This makes none of the equivalences obvious. (More examples of space-filling curves can be found e.g. in \cite{LR12}.)

The property \ref{it:markov} can be seen as a deterministic analogue of the domain Markov property of SLE which O. Schramm defined \cite{Sch00} (i.e. conditioned on an initial segment of the \slek{} trace $\gamma[0,t]$ (in the domain $\HH$), the remaining part of the trace $\gamma[t,\infty]$ is again an \slek{} trace in the domain $\HH \setminus \fill(\gamma[0,t])$). Analogously, the property \ref{it:markov} describes that for any $t$ we have that $\gamma[t,\infty]$, mapped from the domain $\HH \setminus \fill(\gamma[0,t])$ to $\HH$, becomes again a continuous curve.

The property \ref{it:limit} could remind us of \slek{} which are (for some values of $\kappa$) limits of simple curves arising from certain discrete models (e.g. \cite{Smi01,LSW04}). We emphasise that this property is not trivial to show, either. The ``obvious'' attempt to construct an approximating sequence $(\gamma^n)$ would be smoothening the driving function of $\gamma$, but it is not clear whether the produced traces converge uniformly (they only converge in the Carathéodory sense, see \cite[Section 4.7]{Law05}).

Another way of viewing \cref{th:trace_characterisation} is that intuitively Loewner traces are allowed to self-intersect but need to ``bounce-off'' instead of ``crossing over''. But especially when the trace is space-filling, it is not obvious what this means precisely. This theorem describes three equivalent ways of phrasing it.

A consequence of the property \ref{it:markov} is that if we call $\xi$ the driving function of $\gamma$, then $\gamma_t$ is the continuous trace driven by the restriction $\xi_t \defeq \xi\big|_{[t,\infty[}$. To see this, observe that the family of $K_{t,s} \defeq g_t(K_s\setminus K_t)$, $s \ge t$, is the Loewner chain driven by $\xi_t$. It is then easy to see that for each $s \ge t$, we have $\HH \setminus K_{t,s}$ is the unbounded connected component of $\HH \setminus \gamma_t[t,s]$.

In particular, the (pathwise) property of a driving function to generate a continuous trace is a local property.
\begin{corollary}\label{co:driver_restriction}
Suppose $\xi \in C([0,\infty[;\RR)$ generates a trace. Then for any $t \ge 0$, the driving function $\xi_t \defeq \xi\big|_{[t,\infty[} \in C([t,\infty[;\RR)$ generates a trace, namely $\gamma_t$.
\end{corollary}

Again, this statement might ``feel'' obvious to the expert but requires some work to prove. Indeed, D. Zhan has noticed that this statement is not obvious especially for traces with infinitely many self-intersections. The proof would considerably simplify if one only needed to prove that all corresponding hulls are locally connected. But in a discussion with S. Rohde, D. Belyaev noticed that this does not necessarily imply trace continuity, see a counterexample in \cref{fig:trace_not_continuous}.

\begin{remark}\label{re:hcap_need}
In the formulation of \cref{th:trace_characterisation}, there is no need to require the trace to be parametrised by half-plane capacity since the properties do not depend on the parametrisation anyway. But keep in mind that the correspondence between trace and driving function, as in the formulation of \cref{co:driver_restriction}, is defined via half-plane capacity parametrisation (see \cref{sec:preliminaries} for details).

In case $(K_t)$ in \cref{th:trace_characterisation} is parametrised by half-plane capacity, then we can choose $\gamma^n$ parametrised by half-plane capacity as well (since reparametrising does not break the convergence, cf. \cite[Proposition 6.4]{TY20}).
\end{remark}

Another consequence of the property \ref{it:limit} is the following.
\begin{corollary}
The set of chordal Loewner traces parametrised by half-plane capacity is a closed subset of $C([0,\infty[;\barH)$ (with compact-open topology).
\end{corollary}

\begin{remark}
For this statement, some condition on the parametrisation is required, since in general limits of simple traces might fail to be traces (more precisely, the strict monotonicity of the hulls might fail), e.g.
\[ \gamma_n = \begin{cases}
i2t & \text{for } t \in [0,1],\\
i2+(t-1)(1/n-i) & \text{for } t \in [1,2],\\
1/n+i+(t-2) & \text{for } t \ge 2.
\end{cases} \]
Parametrising traces by half-plane capacity prevents such sequences from converging uniformly since the half-plane capacity parametrisation is stable, see e.g. \cite[Proposition 6.3]{TY20}.
\end{remark}

As an application of \cref{th:trace_characterisation}, we give in \cref{se:zero_boundary_time} a simple proof that Loewner traces spend zero ``capacity time'' on the boundary. This statement should be known among experts, but the property \ref{it:limit} considerably simplifies the proof.

\begin{proposition}\label{pr:zero_boundary_time}
Let $\gamma\colon {[0,\infty[} \to \barH$ be a Loewner trace parametrised by half-plane capacity. Then the set $\{ t\ge 0 \mid \gamma(t) \in \RR \}$ has measure $0$.
\end{proposition}

Finally, we discuss again the relationship between trace and driving function. As we have commented above, our proof of property \ref{it:limit} will not involve regularising the driving function of $\gamma$. Instead, we are going to construct $\gamma^n$ in a geometric fashion that does not take the driving function into account. Therefore it is natural to ask what happens to the driving functions during our construction. In fact, we can show that the uniform convergence of traces already implies uniform convergence of their driving functions. Surprisingly, we have not found this explicit statement in the literature. The closest result we have found is \cite[Theorem 4.3]{LMR10}, and indeed we can use almost the same proof to show our claim. The proof will be given in \cref{se:trace_to_driver_continuous}.

\begin{theorem}\label{thm:trace_to_driver_continuous}
Let $\gamma^n \in C([0,\infty[;\barH)$ be a sequence of chordal Loewner traces parametrised by half-plane capacity, with driving functions $\xi^n \in C([0,\infty[;\RR)$. If $\gamma^n \to \gamma$ locally uniformly, then $\xi^n \to \xi$ locally uniformly, where $\xi$ is the driving function of $\gamma$.
\end{theorem}

Note that the map from the trace to its driving function is not uniformly continuous, as the example \cite[Figure 6]{LMR10} shows. Moreover, the converse of \cref{thm:trace_to_driver_continuous} is false, i.e. uniform convergence of driving functions does not imply uniform convergence of their traces, as the example \cite[Example 4.49]{Law05} shows.

One may ask to what extent the approximating sequence in property \ref{it:limit} is unique. Since the left/right turns (in the hyperbolic sense) of a trace are dictated by the increments of its driving function, we see that all $\gamma^n$ will behave similarly in terms of left/right turns. One may also ask for a quantitative description, but we will not investigate it in this paper.

\medskip
\textbf{Acknowledgements:} I would like to thank Steffen Rohde and Fredrik Viklund for helpful comments on earlier versions of the paper. I also thank the referee for their comments.

\section{Preliminaries and Outline}
\label{sec:preliminaries}

We give a brief summary of chordal Loewner chains and traces, and the notation we use in the paper. A compact set $K \subseteq \barH$ such that $\HH \setminus K$ is simply connected is called a compact $\HH$-hull. We identify compact $\HH$-hulls that have the same intersection with $\HH$ (i.e. we distinguish them only by the complementary domains $\HH \setminus K$). We call the mapping-out function of $K$ the unique conformal map $g_K\colon \HH \setminus K \to \HH$ that satisfies the hydrodynamic normalisation $g_K(z) = z+O(\frac{1}{z})$ at $\infty$. The half-plane capacity of $K$ is $\hcap(K) \defeq \lim_{z \to \infty} z(g_K(z)-z) \in [0,\infty[$. For a compact set $A \subseteq \barH$, we define $\fill(A) \subseteq \barH$ to be the union of $A$ with all bounded connected components of $\HH \setminus A$. In case $A$ is connected to $\RR$, this is the smallest compact $\HH$-hull that contains $A$.

A strictly increasing family $(K_t)_{t \ge 0}$ of compact $\HH$-hulls is said to have the local growth property if for any $\varepsilon > 0$ and $T \ge 0$ there exists $\delta > 0$ such that for every $t \in [0,T]$ there exists a crosscut of $\HH \setminus K_t$ of length at most $\varepsilon$ that separates $K_{t+\delta}\setminus K_t$ from $\infty$. When we call $g_t$ the mapping-out function of $K_t$, the local growth property is equivalent to saying that for any $\varepsilon > 0$ and $T \ge 0$ there exists $\delta > 0$ such that $\diam g_t(K_{t+\delta}\setminus K_t) < \varepsilon$ for all $t \in [0,T]$. In particular, the family $(K_{t,s})_{s \ge t}$ with $K_{t,s} \defeq g_t(K_s\setminus K_t)$ again satisfies the local growth property.

For a strictly increasing family $(K_t)_{t \ge 0}$ of compact $\HH$-hulls that satisfies the local growth property, there exists a unique continuous real-valued function $\xi\colon {[0,\infty[} \to \RR$ such that $\xi(t) \in \overline{K_{t,s}}$ for all $0 \le t < s$. This is called the Loewner transform or driving function of $(K_t)_{t \ge 0}$. The correspondence between $(K_t)_{t \ge 0}$ and $\xi$ is one-to-one when we fix the parametrisation of $(K_t)_{t \ge 0}$ in a certain way, e.g. by half-plane capacity, meaning $\hcap(K_t) = 2t$.

A continuous trace is a continuous path $\gamma\colon {[0,\infty[} \to \barH$ with $\gamma(0) \in \RR$ such that the family $\fill(\gamma[0,t])$ satisfies the local growth property. We say that $\xi\in C([0,\infty[;\RR)$ generates a continuous trace if there exists such $\gamma$ that is parametrised by half-plane capacity and has $\xi$ as driving function, which is equivalent to saying that the limit $\gamma(t) = \lim_{y \searrow 0} g_t^{-1}(iy+\xi(t))$ exists for all $t$ and is continuous in $t$. A trace is called simple if it intersects neither itself nor $\RR \setminus \{\gamma(0)\}$.

When we have two traces $\gamma^1\colon [0,t_1] \to \barH$ and $\gamma^2\colon [t_1,t_2] \to \barH$, we can glue them to a trace $\gamma(s) = \gamma^1(s)$ on $[0,t_1]$ and $\gamma(s) = g_{t_1}^{-1}(\gamma^2(s)-\gamma^2(t_1)+\xi(t_1))$ on $[t_1,t_2]$, and the driving function of $\gamma$ is the concatenation of $\xi^1$ and $\xi^2$. The converse statement is \cref{co:driver_restriction} which we will prove in this paper.

\subsection{Outline}

We give a few comments and first steps on the proof of \cref{th:trace_characterisation}.

The fact that \ref{it:limit} implies \ref{it:lgp} has been shown in \cite[Proposition 6.3]{TY20}. The converse statement, i.e. \ref{it:lgp} implies \ref{it:limit}, is proven in \cref{se:trace_approximation}. For that part we will also make use of the property \ref{it:markov} which we will show first (below and in \cref{se:trace_excursions}).

The fact that \ref{it:markov} implies \ref{it:lgp} follows almost immediately from \cite[Lemma 4.5]{LMR10}. One has to observe that although the lemma is formulated for connected sets $S$, its proof shows that it suffices when $\overline{g(S)}$ is connected. In particular, when we assume $\gamma_t$ to be continuous, the lemma can be applied to
\[ \diam g_t(\gamma[t,t+\delta]) \le c\sqrt{\diam\gamma[t,t+\delta]} . \]
With the uniform continuity of $\gamma$, the local growth property follows.

For the proof that \ref{it:lgp} implies \ref{it:markov}, we gather a few preliminary observations. The continuity of $\gamma$ tells us an important piece of information about $\gamma_t$. Recall the following statement which follows from \cite[Theorem 1.7]{Pom92} via a Möbius transformation taking $z \in \partial H$ to $\infty$.
\begin{lemma}
Let $f\colon \DD \to H \subseteq \hatC$ be conformal, and $z \in \partial H$. Then the set $f^{-1}(z) \subseteq \partial \DD$ has measure $0$. 
\end{lemma}

\begin{corollary}
Let $s \ge t$. The set of limit points of $\gamma_t$ at $s$ is a single point or a subset of $\RR$ with measure $0$.
\end{corollary}

Since $\gamma$ is continuous, all $K_t$ are locally connected, and hence $\gamma_t$ is right-continuous, and is continuous at times where it is in $\HH$. It follows that $\gamma_t$ consists of a countable number of excursions in $\HH$ from $\RR$. Together with the previous observation, we conclude the following.
\begin{lemma}
For any $\delta > 0$, there are finitely many excursions of $\gamma_t$ with diameter greater than $\delta$ on finite time intervals.
\end{lemma}

\begin{proof}
Suppose there are infinitely many excursions of $\gamma_t$ with diameter greater than $\delta$ on some finite time interval $[t,T]$. Since $\gamma_t$ is bounded, by compactness of the Hausdorff metric (see \cite[Theorem 3.2.4]{Bee93}) we can find a sequence of excursions $\tilde\gamma_n$ (considered as compact sets in $\barH$) that converge in the Hausdorff metric to a compact set $A \subseteq \barH$, and $A$ is connected (see \cite[Exercise 3.2.8]{Bee93}). We can choose the sequence such that also the occurring times of $\tilde\gamma_n$ converge to some $\bar s \in [t,T]$. Then all points in $A$ are limit points of $\gamma_t$ at $\bar s$, and therefore a single point or a subset of $\RR$ with measure $0$. Since $A$ is connected, it must be a single point, contradicting $\diam(\tilde\gamma_n) > \delta$.
\end{proof}

It follows easily that $K_{t,s}$ is locally connected for each $s \ge t$ (see \cref{thm:beta_locally_connected}). Note that this is not enough to show that $\gamma_t$ is continuous, as the following variation of an example by D. Belyaev in \cref{fig:trace_not_continuous} shows.

\begin{figure}[h]
	\centering
	\includegraphics[width=0.48\textwidth]{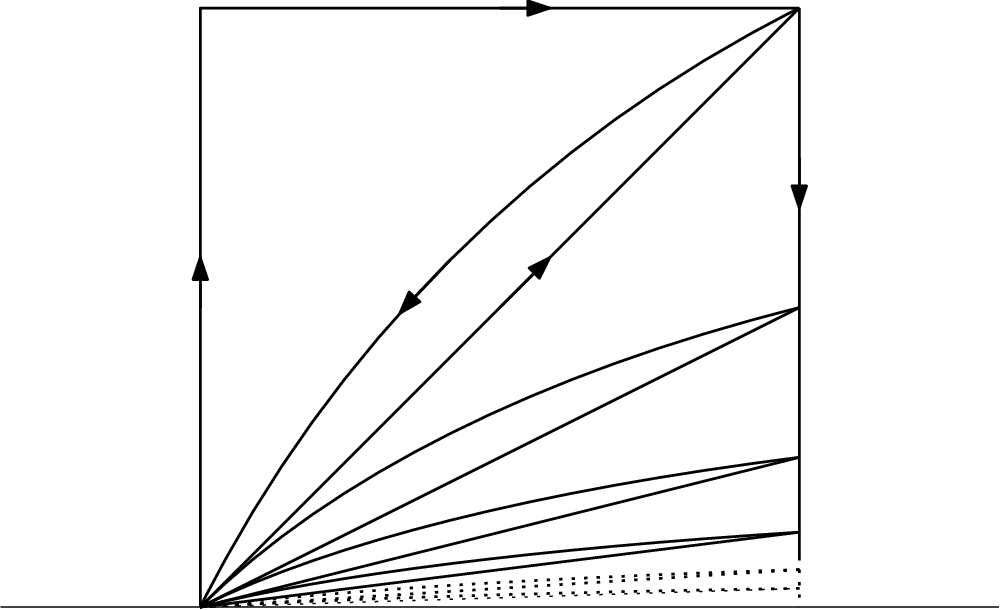}
	\caption{All hulls are locally connected, but the ``trace'' is not continuous. Variation of an example by D. Belyaev \cite[p. 212]{Bel20}.}
	\label{fig:trace_not_continuous}
\end{figure}

Observe that in the above ``non-example'' there are infinitely many large excursions. We show in \cref{se:trace_excursions} that all counterexamples look like this, and hence do not apply to $\gamma_t$. This will establish the continuity of $\gamma_t$.

For the convenience of the reader we recall two classical results about the topology of the plane. See \cite[Section 1.5]{Pom75} for proofs.

\begin{theorem}[Janiszewski]
Let $A_1,A_2 \subseteq \hatC$ be closed sets such that $A_1 \cap A_2$ is connected. If two points $a,b \in \hatC$ are neither separated by $A_1$ nor by $A_2$, then they are not separated by $A_1 \cup A_2$.
\end{theorem}

\begin{theorem}[Jordan curve theorem]
If $J \subseteq \hatC$ is a simple loop, then $\hatC \setminus J$ has exactly two components $G_0$ and $G_1$, and these satisfy $\partial G_0 = \partial G_1 = J$.
\end{theorem}

\section{Excursions of Loewner traces}
\label{se:trace_excursions}

In the following, we assume that $\beta\colon {[0,\infty[} \to \barH$ has the following properties (we do not a priori assume $\beta$ to be a continuous function):
\begin{itemize}
\item $\beta$ consists of (a countable number of) excursions in $\HH$, i.e. for each $t \ge 0$ if $\beta(t) \in \HH$, then there exist $t_1 < t < t_2$ such that $\beta$ is continuous on $]t_1,t_2[$, has limits $\beta(t_1-), \beta(t_2+) \in \RR$, and $\beta(]t_1,t_2[) \subseteq \HH$.
\item For each $T \ge 0$ and $\delta > 0$ there exist only finitely many excursions of $\beta$ on the time interval $[0,T]$ with diameter greater than $\delta$.
\item For $t \ge 0$, $K_t \defeq \fill(\beta[0,t])$ are compact, strictly increasing, and satisfy the local growth property.
\end{itemize}
With a slight abuse of notation, an excursion $\tilde\beta$ of $\beta$ will denote either the path $\tilde\beta \in C([t_1,t_2];\barH)$ or the set $\tilde\beta[t_1,t_2] \subseteq \barH$ (where $\tilde\beta(t_1)$ and $\tilde\beta(t_2)$ denote the limit points $\beta(t_1-)$, $\beta(t_2+)$). As usual, we write $H_t \defeq \HH \setminus K_t$.

Observe that the strict monotonicity of $(K_t)$ implies that the set of times that belong to excursions is dense. Moreover, the local growth property implies that $K_t \cap \RR$ is an interval for every $t$.

Observe also that for $z \in \HH$, we have $z \in K_t$ if and only if $z$ lies on or is separated from $\infty$ by some excursion until time $t$. This is because only finitely many excursions have diameter larger than $\Im z$.

The main goal of this section is to show the following.
\begin{proposition}\label{thm:beta_continuous}
The path $\beta$ is continuous in the sense that for every sequence $t_n \to t$ such that $\beta(t_n)$ is on some excursion, the limit $\lim_{n \to \infty} \beta(t_n)$ exists.\\
(Equivalently, $\beta$ can be extended to a continuous function from $[0,\infty[$ to $\barH$.)
\end{proposition}
Note that from our assumptions on $\beta$, it does not make sense to specify $\beta(t)$ at times $t$ where $\beta$ is not on any excursion.

\begin{lemma}\label{thm:beta_locally_connected}
For each $t \ge 0$, $K_t$ is locally connected.
\end{lemma}

\begin{proof}
For $z \in \HH$, $K_t$ is clearly locally connected at $z$ since only finitely many excursions intersect $z$.

For $z \in \RR$, let $\delta > 0$. There are only finitely many excursions of diameter at least $\delta$ until time $t$. Call $K$ the union of the fillings of these excursions. Then there exists a connected set $A_1 \subseteq K \cap B(z,\delta)$ that contains $K \cap B(z,r)$ for some $r > 0$. Consider the set
\[ A_2 \defeq A_1 \cup \bigcup \{ \fill(\tilde\beta) \mid \tilde\beta \text{ is an excursion with } \diam \tilde\beta < \delta \text{ and } \dist(z,\tilde\beta) < r \} \]
which is a connected set contained in $K_t \cap B(z,\delta+r)$. Then $K_t \cap B(z,r) \setminus A_2$ can only consist of connected components of $B(z,r) \setminus A_2$ since all excursions that intersect $B(z,r)$ have been included in $A_2$. Therefore $A_3 \defeq (K_t \cap B(z,r)) \cup A_2 = A_2 \cup (K_t \cap B(z,r) \setminus A_2)$ is a connected set within $K_t \cap B(z,\delta+r)$ that contains $K_t \cap B(z,r)$. This shows local connectedness at $z$.
\end{proof}

\begin{lemma}\label{thm:lc_finitely_many_ball_components}
Let $D \subseteq \hatC$ be a domain with locally connected boundary. Let $z \in \CC$ and $0 < r_1 < r_2$. Then only finitely many components of $D \cap B(z,r_1)$ are disconnected in $D \cap B(z,r_2)$.
\end{lemma}

\begin{proof}
Let $r' \defeq \frac{r_1+r_2}{2}$. If $D \subseteq B(z,r_2)$, there is nothing to prove. Therefore we can suppose there is some $z_0 \in D \setminus \overline{B(z,r_2)}$. For every $z' \in D \cap B(z,r_1)$ we can find a simple polygonal path $\alpha_{z'}$ in $D$ from $z'$ to $z_0$. Note that such paths hit any circle only finitely many times. Pick $\alpha_{z'}$ such that it hits $\partial B(z,r')$ as few times as possible.

Suppose that there exist infinitely many $z' \in D \cap B(z,r_1)$ that are disconnected in $D \cap B(z,r_2)$. Let $A$ be an infinite set of such $z'$. For $z' \in A$ the paths $\alpha_{z'}$ are all disjoint in $B(z,r_2)$. Denote by $w_{z'}$ the first hitting point of $\alpha_{z'}$ with $\partial B(z,r')$. Then $B = \{ w_{z'} \mid z' \in A \}$ is an infinite set and hence has a limit point $w_0 \in \partial B(z,r')$.

Clearly $w_0 \in \partial D$ since all points in $B$ are disconnected in $D \cap B(z,r_2)$ by construction. Since $\partial D$ is locally connected, we can find a connected set $C \subseteq \partial D \cap B(w_0,r'-r_1)$ that contains $\partial D \cap B(w_0,2\delta)$ for some $\delta > 0$. Then each two points in $D$ that are connected in $B(w_0,2\delta) \setminus C$ are also connected in $D$. Let $w_{z'} \in B \cap B(w_0,\delta)$. We claim that $\alpha_{z'}$ needs to pass a segment of $\partial B(w_0,\delta) \setminus C$ that intersects $\partial B(z,r')$. This gives us the desired contradiction since there are only two such segments but infinitely many points in $B \cap B(w_0,\delta)$.

Note that $\alpha_{z'}$ needs to enter $B(w_0,\delta)$ through a segment $S$ of $\partial B(w_0,\delta) \setminus C$ before passing $w_{z'}$. We show below that it needs to cross $S$ again. If $S$ does not intersect $\partial B(z,r')$, then $w_{z'}$ is an unnecessary crossing of $\partial B(z,r')$ which contradicts our construction.

Suppose that $\alpha_{z'}$ does not pass $S$ again, which implies that it crosses $S$ an odd number of times. Let $\zeta_1, \zeta_2 \in \partial B(w_0,\delta) \cap C$ be the endpoints of $S$. We show that $\zeta_1$ and $\zeta_2$ cannot be connected in $C$ which contradicts the connectedness of $C$. Consider the segment of $\alpha_{z'}$ from when it last enters $B(w_0,r'-r_1)$ until it next leaves $B(w_0,r'-r_1)$ (these times exist since $\alpha_{z'}$ begins inside $B(z,r_1)$ and ends outside $B(z,r_2)$), followed by an arc of $\partial B(w_0,r'-r_1)$. The Jordan curve theorem then implies that any set that connects $\zeta_1$ and $\zeta_2$ in $\hatC \setminus \alpha_{z'}$ needs to intersect $\partial B(w_0,r'-r_1)$. But $C$ cannot do this because $C \subseteq B(w_0,r'-r_1)$.
\end{proof}

Intuitively, the local growth property implies that $\beta$ might touch but not cross itself again. In particular, it cannot cross any of its past excursions. We make this more precise in the following.

For $h > 0$, we write $\mathcal{S}_h \defeq \{ z \in \CC \mid \Im z \in {]0,h[} \}$.

Let $K \subseteq \barH$ be a compact $\HH$-hull and $h > 0$. We say that two points in $\mathcal{S}_h \setminus K$ are \emph{on the same $h$-side of} $K$ if they are connected in $\mathcal{S}_{h'} \setminus K$ for every $h' > h$. See \cref{fig:hull_sides} for an illustration of this definition.

Note that if $h$ is smaller than the height of $K$, then $K$ has at least two $h$-sides (a left and a right side). If $K_1 \subseteq K_2$, then points on the same $h$-side of $K_2$ are also on the same $h$-side of $K_1$.

\begin{figure}[h]
	\centering
	\includegraphics[width=0.48\textwidth]{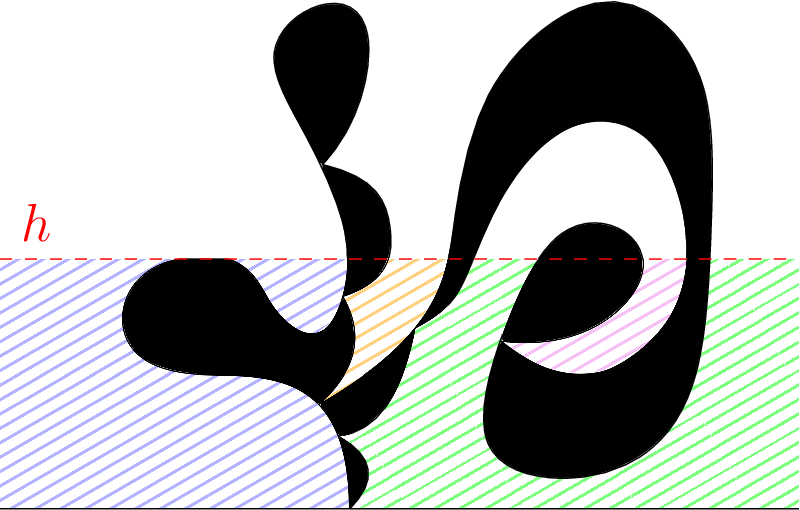}
	\caption{A hull with four $h$-sides.}
	\label{fig:hull_sides}
\end{figure}

\begin{lemma}\label{thm:cut_off_sides}
Let $K \subseteq \barH$ be a compact $\HH$-hull, and $h>0$. Fix two different $h$-sides $S_1,S_2$ of $K$. Then there exists $\delta > 0$ with the following property:

If $C$ is a crosscut in $\HH \setminus K$ such that there exist $z_1 \in S_1$ and $z_2 \in S_2$ that both are separated from $\infty$ by $C$, then $\diam C \ge \delta$.
\end{lemma}

\begin{proof}
Since $S_1$ and $S_2$ are different $h$-sides of $K$, there exists $h' > h$ such that they are disconnected in $\mathcal{S}_{h'} \setminus K$.

Let $h'' \in {]h,h'[}$. By definition, all points in $S_1$ are connected in $\mathcal{S}_{h''} \setminus K$. Pick any $z \in S_1$. Since $\HH \setminus K$ is a domain, there exists a path $\alpha$ in $\HH \setminus K$ from $z$ to a neighbourhood of $\infty$. Therefore any crosscut that separates $z$ from $\infty$ needs to cross $\alpha$. It follows that any crosscut that separates some point in $S_1$ from $\infty$ needs to cross either $\alpha$ or some point connected to $S_1$ in $\mathcal{S}_{h''} \setminus K$. Let $\delta_1 \defeq \dist(\alpha,\partial K) > 0$. Then any crosscut with diameter smaller than $\delta_1$ that separates some point in $S_1$ from $\infty$ needs to contain some point connected to $S_1$ in $\mathcal{S}_{h''} \setminus K$. Similarly, there is $\delta_2$ such that the analogous statement is true for $S_2$.

Now let $\delta \defeq \delta_1 \wedge \delta_2 \wedge (h'-h'')$. If $C$ is a crosscut in $\HH \setminus K$ with $\diam C < \delta$ and separates points both in $S_1$ and $S_2$ from $\infty$, then $C$ minus its endpoints is a connected set in $\mathcal{S}_{h'} \setminus K$ that contains two points connected to $S_1$ resp. $S_2$ in $\mathcal{S}_{h''} \setminus K$. But this is impossible since $S_1$ and $S_2$ are separated in $\mathcal{S}_{h'} \setminus K$.
\end{proof}

\begin{corollary}\label{thm:crosscut_intersect_new_points}
Let $K \subseteq \barH$ be a compact $\HH$-hull, and $h>0$. Fix two different $h$-sides $S_1,S_2$ of $K$. Then there exists $\delta > 0$ with the following property:

If $K' \supseteq K$ is a compact $\HH$-hull and $C$ is a crosscut in $\HH \setminus K'$ with $\diam C < \delta$ such that there exist $z_1 \in S_1 \setminus K'$ and $z_2 \in S_2\setminus K'$ that both are separated from $\infty$ by $C$, then $C$ intersects $K' \setminus K$.
\end{corollary}

\begin{proof}
Choose $\delta$ as in \cref{thm:cut_off_sides}. If $C$ does not intersect $K' \setminus K$, then $C$ is also a crosscut in $\HH \setminus K$. We claim that $C$ separates $z_1$, $z_2$ from $\infty$ also in $\HH \setminus K$ which is a contradiction to $\diam C < \delta$.

Suppose $C \cup K \cup \hat\RR$ does not separate $z_1$ from $\infty$. Since $K' \cup \hat\RR$ does not separate $z_1$ from $\infty$ either and $(C \cup K \cup \hat\RR) \cap (K' \cup \hat\RR) = K \cup \hat\RR$ is connected (recall that we assumed $C \cap K' \subseteq K$), by Janiszewski's theorem $(C \cup K \cup \hat\RR) \cup (K' \cup \hat\RR) = C \cup K' \cup \hat\RR$ would not separate $z_1$ from $\infty$, which contradicts our assumption. The argumentation for $z_2$ is the same.
\end{proof}

\begin{figure}[h]
	\centering
	\includegraphics[width=0.48\textwidth]{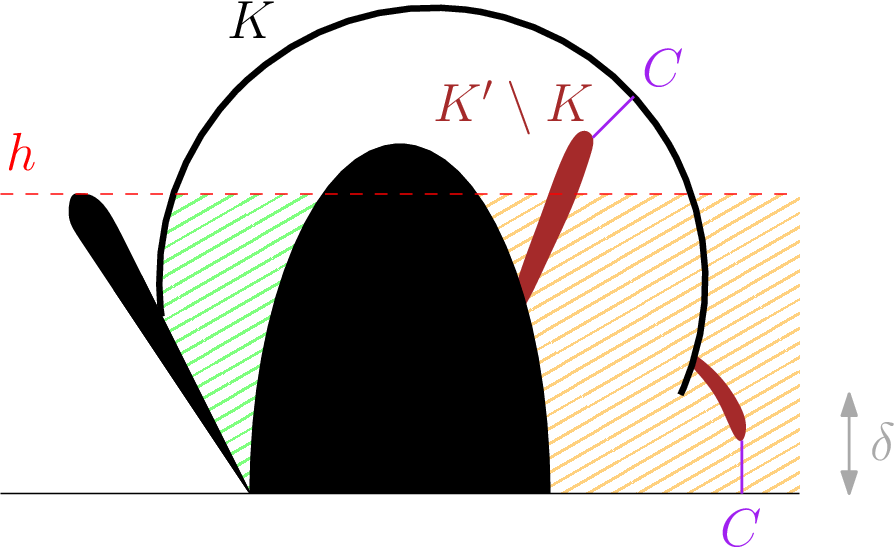}
	\caption{The situation in \cref{thm:crosscut_intersect_new_points}.}
\end{figure}

We say that an excursion $\tilde\beta \in C([t_1,t_2];\barH)$ \emph{occurs within a time interval} $[s,t] \subseteq \RR$ if ${]t_1,t_2[} \cap [s,t] \neq \varnothing$.

Let $K \subseteq \barH$ be a compact $\HH$-hull and $h > 0$. We say that $\beta[s,t]$ is \emph{on one $h$-side of} $K$ if all points of $\beta[s,t] \cap (\HH \setminus K)$ lie on the same $h$-side of $K$.

\begin{lemma}\label{thm:beta_continues_on_one_side}
Let $t \ge 0$ and $h>0$. If $\Im \beta(t) < h$, then $\beta[t,t+\varepsilon]$ is on one $h$-side of $K_t$ for some $\varepsilon > 0$.
\end{lemma}

\begin{proof}
By compactness we can find a sequence $t_n \searrow t$ such that $\beta(t_n) \in H_t$ converges to some $z \in \barH$ with $\Im z < h$. By \cref{thm:lc_finitely_many_ball_components} only finitely many components of $H_t \cap B(z,(h-\Im z)/2)$ are disconnected in $H_t \cap B(z,h-\Im z)$. Therefore (by the pigeonhole principle) we can pick a subsequence of $(t_n)$ (call it $(t_n)$ again) such that all $\beta(t_n)$ are connected in $H_t \cap B(z,h-\Im z) \subseteq \mathcal{S}_h \setminus K_t$. In particular, they are all on the same $h$-side of $K_t$; call that side $S_1$.

Suppose that there is another sequence $s_n \searrow t$ such that each $\beta(s_n) \in H_t$ is on a different $h$-side of $K_t$ than $S_1$. By the same argument as above, we can pick the sequence such that all $\beta(s_n)$ are on the same $h$-side of $K_t$; call that side $S_2$.

By construction $S_1 \neq S_2$. But then \cref{thm:cut_off_sides} gives us a contradiction to the local growth property.
\end{proof}

\begin{lemma}\label{thm:beta_continues_on_same_side}
Let $0 \le s < t$ and $h>0$. If $\beta[s,t]$ is on one $h$-side of $K_s$ and $\Im \beta(t) < h$, then $\beta[s,t+\varepsilon]$ is on one $h$-side of $K_s$ for some $\varepsilon > 0$.
\end{lemma}

\begin{proof}
If $\beta(t) \in H_s$, then by the continuity of excursions there is nothing to show, so assume $\beta(t) \in K_s \cup \RR$. We claim that the set of limit points $\beta(t-)$ is contained in $K_s \cup \RR$. In case $\beta(t) \in \HH$, this is clear by the continuity of excursions. In case $\beta(t) \in \RR$ we have either $t$ as a finishing time of an excursion (in which case the claim is again clear by continuity) or that there are infinitely many excursions finishing shortly before $t$ in which case their diameters have to converge to $0$ by the assumption on $\beta$ which implies the claim.

Call $S_1$ the $h$-side of $K_s$ containing $\beta[s,t]$. By \cref{thm:beta_continues_on_one_side}, $\beta[t,t+\varepsilon]$ is on one $h$-side of $K_t$ and hence also of $K_s$ for some $\varepsilon > 0$; call it $S_2$. Suppose $S_1 \neq S_2$. Then we can find $h' > h$ such that they are separated in $\mathcal{S}_{h'} \setminus K_s$.

Pick a sequence $t_n \searrow t$ such that $\beta(t_n) \in H_t$. As just observed, we have $\beta(t_n) \in S_2$. Pick any $t_N$ and find a path $\alpha$ in $H_t$ connecting $\beta(t_N)$ to a neighbourhood of $\infty$. We have seen that the set of limit points $\beta(t-)$ is contained in $K_t$, so $\delta \defeq \dist(\alpha, \beta(t-)) > 0$. Find $t' < t$ such that $\dist(\beta(t''),\beta(t-)) < \delta/2$ for all $t'' \in {[t',t[}$. 

Since $\beta[s,t]$ is on one $h$-side of $K_s$, it follows from Janiszewski's theorem that $\beta[t',t]$ is on one $h$-side of $K_{t'}$. Recall that we have chosen all $\beta(t_n)$ to be on one different $h$-side of $K_{t'}$. Applying \cref{thm:crosscut_intersect_new_points} to $K_{t'}$ and by the local growth property there exists some $t'' \in {[t',t[}$ and some crosscut $C$ in $H_{t''}$ with $\diam C < \delta/2 \wedge (h'-h)$ that separates $\beta(t_n)$ from $\infty$ for sufficiently large $n$ and intersects $K_{t''} \setminus K_{t'}$.

The choice of $\delta$ implies $\dist(C,\alpha) > 0$. Therefore $C$ does not separate $\beta(t_N)$ from $\infty$. We claim that $C$ does not separate $\beta(t_n)$ from $\infty$ for any $n$, producing a contradiction.

We have picked $t_n$ such that all $\beta(t_n)$ are connected in $\mathcal{S}_{h''} \setminus K_t$ for any $h'' > h$. If $C$ separates $\beta(t_n)$ from $\infty$, $C$ needs to contain some point in the same $h$-side of $K_t$ as $\beta(t_n)$, and that side is contained in $S_2$. This means that $C$ needs to contain points from both $S_1$ and $S_2$. Since all points of $C$ are less than $h'-h$ away from the set $\beta(t-)$, $C$ contains a connected set in $\mathcal{S}_{h'} \setminus K_s$. But this is impossible since $S_1$ and $S_2$ are separated in $\mathcal{S}_{h'} \setminus K_s$.
\end{proof}

\begin{corollary}\label{thm:side_not_crossed}
Let $0 \le s < t$ and $h>0$. If all excursions of $\beta$ that occur within the time interval $[s,t]$ have smaller diameter than $h$, then $\beta[s,t]$ lies on one $h$-side of $K_s$.
\end{corollary}

\begin{proof}
Let
\[ \bar t \defeq \sup\{ t' \ge s \mid \beta[s,t'] \text{ is on one $h$-side of } K_s \} . \]
By \cref{thm:beta_continues_on_one_side}, we have $\bar t > s$, and by \cref{thm:beta_continues_on_same_side}, we have $\bar t \ge t$.
\end{proof}

Now the proof of \cref{thm:beta_continuous} follows.
\begin{proof}[Proof of \cref{thm:beta_continuous}]
First we show left-continuity. Let $t \ge 0$. If some excursion is ongoing or finishes at $t$, then there is nothing to show. Therefore assume that there are infinitely many excursions of $\beta$ finishing shortly before $t$.

Recall that $I \defeq K_t \cap \RR$ is an interval. Hence for any $x \in I$, there exists some past excursion $\tilde\beta$ such that $\fill(\tilde\beta)$ has small distance to $x$. Let $h>0$ be smaller than the height of $\tilde\beta$. From \cref{thm:side_not_crossed} and the assumption that only finitely many excursions are larger than $h$, it follows that when $\varepsilon > 0$ is small enough, $\beta[t-\varepsilon,t]$ will lie on one $h$-side of $K_{t-\varepsilon}$ and hence also of $\fill(\tilde\beta)$. Since this holds for all $x \in I$, it implies that $\beta(t-)$ is a Cauchy sequence.

Now let $x$ be any right limit point of $\beta$. If $x \neq \beta(t-)$, then as above we can find some past excursion between $x$ and $\beta(t-)$, contradicting \cref{thm:beta_continues_on_same_side}.
\end{proof}

\section{Proof of \ref{it:limit} in \cref{th:trace_characterisation}}
\label{se:trace_approximation}

Since this part is about local convergence, we can restrict ourselves to a compact time interval, say $[0,1]$. Let $\gamma \in C([0,1];\barH)$ be a trace. The strategy is to insert a sequence of cut points into $\gamma$ at a countable dense subset of $[0,1]$. This will produce a simple trace that approximates $\gamma$.

For $\gamma \in C([0,1];\barH)$ that satisfies the local growth property, we denote by $\hat g_t\colon \HH \setminus \fill(\gamma[0,t]) \to \HH$ the conformal map with $g_t(\gamma(t+)) = 0$ and $g_t(z) = z+O(1)$ near $\infty$, and $\hat f_t \defeq \hat g_t^{-1}$. In this section, we write $\gamma_t(s) \defeq \hat g_t(\gamma(s))$ for $t \le s \le 1$. By the property \ref{it:markov} of \cref{th:trace_characterisation}, this is again a continuous trace (generated by $\xi_t(s) \defeq \xi(s)-\xi(t)$, $s \in [t,1]$). Note the re-centring here which is a slight change of notation to the previous sections.

We first sketch how we construct a sequence $(\gamma^n)$ that converge to a simple path $\gamma^\infty$ such that $\|\gamma-\gamma^\infty\|_\infty < \varepsilon$. To keep the notation a bit simpler, we will care only about $\gamma^\infty$ being simple and not about boundary hittings. The latter are not a problem since we can remove them via
\[ \tilde\gamma^\infty \defeq \begin{cases}
\gamma^\infty(0)+i2\sqrt{t} & \text{for } t \le \varepsilon,\\
\gamma^\infty(t-\varepsilon)+i2\sqrt{\varepsilon} & \text{for } t \ge \varepsilon.
\end{cases} \]

Let $(t_n)$ be a sequence such that $\{ t_n \mid n \in \NN \}$ is a dense subset of $[0,1]$. Each $\gamma^n$ will insert a short simple path into $\gamma$ which serves as cut points. This path will be inserted in the time interval $[t_n,t_n+h_n]$ for some small $h_n > 0$. As a result, all times $t>t_n$ will shift to $t+h_n$. Therefore it is notationally convenient to introduce another (slight) reparametrisation.

Suppose a summable sequence of $h_n > 0$ have been defined, and write $\bar h \defeq \sum_{n \in \NN} h_n$. We ``stretch'' the interval $[0,1]$ to $[0,1+\bar h]$ by inserting an additional interval $[t_n,t_n+h_n]$ at time $t_n$ for each $n$. More precisely, we define $\varphi\colon [0,1] \to [0,1+\bar h]$,
\[ \varphi(t) \defeq t + \sum_{m \in \NN \text{ s.th. } t_m < t} h_m . \]
Let $s_n \defeq \varphi(t_n)$ and $I_n \defeq [s_n,s_n+h_n] \subseteq [0,1+\bar h]$. Then
\[ \begin{split}
\varphi^{-1}(s) \defeq{}& \sup\{ t \in [0,1] \mid \varphi(t) \le s \} \\
={}& \begin{cases}
s-\displaystyle\sum_{m \in \NN \text{ s.th. } s_m < s} h_m & \text{ if } s \notin \bigcup_n I_n ,\\
t_n & \text{ if } s \in I_n \text{ for some } n .
\end{cases}
\end{split} \]

We will construct $\gamma^n \in C([0,1+\bar h];\barH)$ inductively. Let $\gamma^0$ be $\gamma$ but ``halted'' in the intervals $I_n$, i.e. $\gamma^0(s) \defeq \gamma(\varphi^{-1}(s))$. Note that the hulls generated by $\gamma^0$ are not strictly increasing (they remain constant in the intervals $I_n$), but this will not worry us because we will construct $\gamma^\infty$ to be strictly increasing.

For $n \ge 1$, we let (see \cref{fi:cut_interval})
\[ \gamma^n(s) \defeq \begin{cases}
\gamma^{n-1}(s) & \text{for } s \le s_n,\\
\hat f^{n-1}_{s_n}(i2\sqrt{s-s_n}) & \text{for } s \in I_n,\\
\hat f^{n-1}_{s_n}(i2\sqrt{h_n}+\gamma^{n-1}_{s_n}(s)) & \text{for } s \ge s_n+h_n.
\end{cases} \]

\begin{figure}[h]
	\centering
	\includegraphics[width=0.8\textwidth]{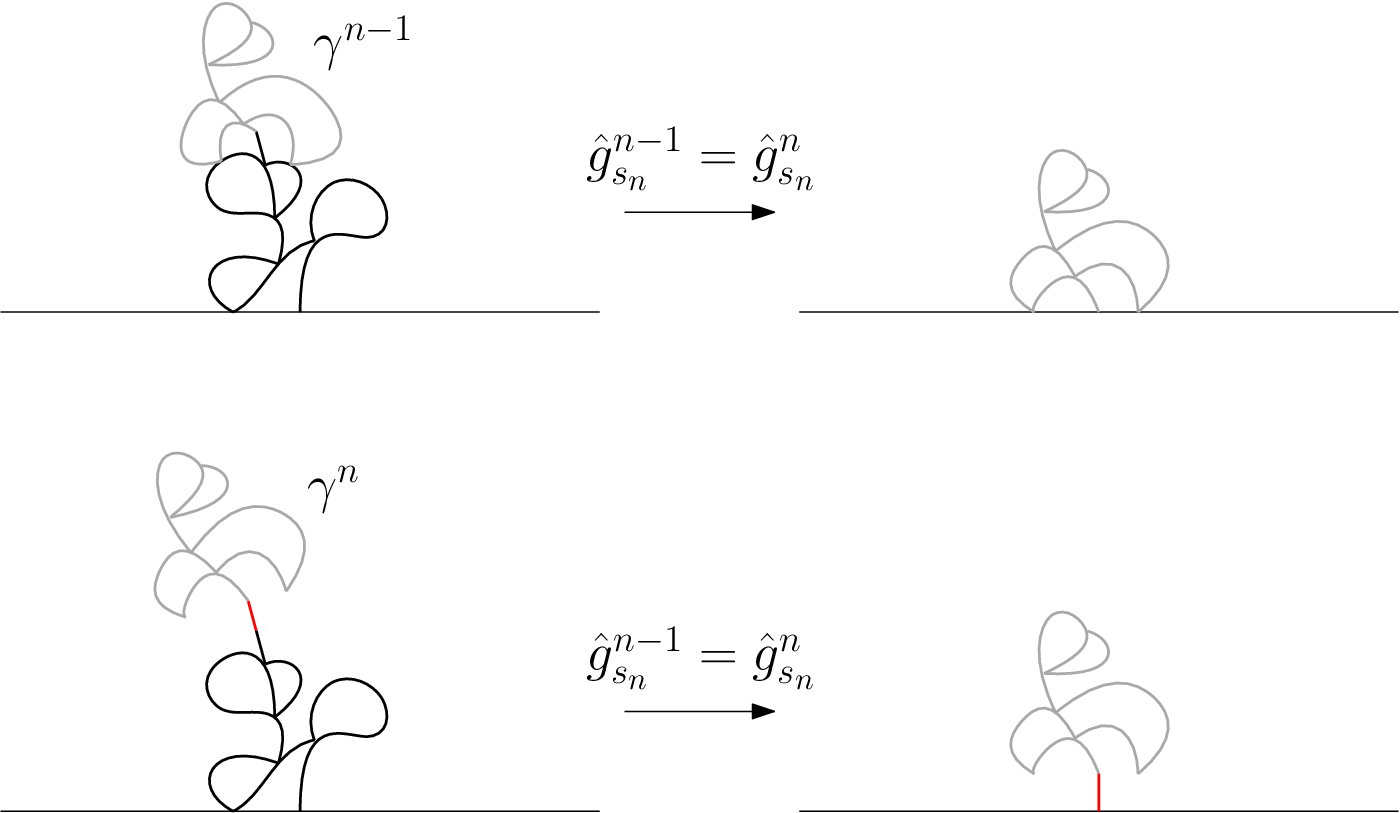}
	\caption{The construction of $\gamma^n$ from $\gamma^{n-1}$.}
	\label{fi:cut_interval}
\end{figure}

We claim that $\gamma^n$ satisfies the local growth property again. For $s \le s_n+h_n$ this is clear. For $s \ge s_n+h_n$ it follows from the local growth property of $\gamma^{n-1}_{s_n}$. (More precisely, for each crosscut $C$ in $\HH \setminus \fill(\gamma^{n-1}_{s_n}[s_n,s])$, we can build a crosscut in $\HH \setminus \fill(\gamma^{n}_{s_n}[s_n,s])$ by $\tilde C \defeq C+i2\sqrt{h_n}$ and closing $\tilde C$ from below in case $C$ terminates on $\RR$.)

Note that we have inserted a ``cut segment'' in the interval $I_n$ which separates $\gamma^n[s_n+h_n,1+\bar h]$ from $\gamma^n[0,s_n]$. We would like to make sure that these two parts remain separated for $m > n$, therefore we introduce the following notation.

For $m \ge n$, we let $d_{n,m} \defeq \dist(\gamma^m[0,s_n], \gamma^m[s_n+h_n,1+\bar h])$. We will show later that we can pick the sequences $(t_n)$, $(h_n)$ such that the following conditions are satisfied.
\begin{itemize}
\item $\|\gamma^n-\gamma^{n-1}\|_\infty < \varepsilon 2^{-n}$.
\item $d_{n,m} > d_{n,n}/2 > 0$ for all $m>n$.
\item $|\gamma^m(s)-\gamma^m(s')| > \frac{1}{2}|\gamma^n(s)-\gamma^n(s')|$ for $s,s' \in I_n$ and $m>n$.
\end{itemize}

These conditions will imply that $\gamma^n \to \gamma^\infty$ for some $\gamma^\infty \in C([0,1+\bar h];\HH)$ with $\|\gamma^0-\gamma^\infty\|_\infty < \varepsilon$. Moreover, we show that $\gamma^\infty$ is simple. Let $0 \le s < s'$. We need to show that $\gamma^\infty(s) \neq \gamma^\infty(s')$. There are two cases. In case there exists some $n$ such that $s < s_n < s_n+h_n < s'$, then $|\gamma^m(s)-\gamma^m(s')| \ge d_{n,m} > d_{n,n}/2$ for $m>n$ and hence $|\gamma^\infty(s)-\gamma^\infty(s')| \ge d_{n,n}/2 > 0$. In case no such $n$ exists, by the denseness of the sequence $(t_n)$, we must have $s,s' \in I_n$ for some $n$. In that case we have $|\gamma^m(s)-\gamma^m(s')| > \frac{1}{2}|\gamma^n(s)-\gamma^n(s')|$ for $m>n$ and hence $|\gamma^\infty(s)-\gamma^\infty(s')| > \frac{1}{2}|\gamma^n(s)-\gamma^n(s')| > 0$.

Now, since $\gamma^0$ is just a time-changed version of $\gamma$ (by at most $\bar h$), the uniform continuity of $\gamma$ implies that $\|\gamma-\gamma^\infty\|_\infty < \varepsilon+\phi(\bar h)$ for some increasing function $\phi$ with $\phi(0+)=0$.

This finishes the proof of \ref{it:limit} of \cref{th:trace_characterisation} since $\varepsilon$ and $\bar h$ can be chosen arbitrarily small.

\medskip

It remains to find suitable sequences $(t_n)$, $(h_n)$ that satisfy our desired conditions.

\begin{lemma}\label{thm:countable_dense_times}
There exists a countable dense subset $T \subseteq [0,1]$ such that for each $t \in T$ we have $\gamma(t) \notin \gamma([0,t[) \cup \RR$.
\end{lemma}

\begin{proof}
Since the family $(K_t)$ is strictly increasing, there must exist such $t$ in every interval of positive length. The claim follows immediately.
\end{proof}

\begin{lemma}\label{thm:small_distance_perturbation}
Let $f\colon \HH \to D \subseteq \CC$ be a conformal map and $A \subseteq \HH$ a bounded set with $\dist(A,\RR) > 0$. Then for any $\varepsilon > 0$ there exists $\delta > 0$ such that $|f(z_1+ih)-f(z_2+ih)| \ge (1-\varepsilon)|f(z_1)-f(z_2)|$ for all $z_1,z_2 \in A$ and $h \in [0,\delta]$.
\end{lemma}

\begin{proof}
Let $d>0$ be a small number that we specify later. Since $f$ is uniformly continuous on a neighbourhood of $A$, there certainly exists $\delta > 0$ that work for all $z_1,z_2 \in A$ with $|z_1-z_2| \ge d^2$.

Suppose now that $|z_1-z_2| < d^2$. We can assume that $d < \frac{1}{2} \dist(A,\RR)$. The Koebe distortion theorem and Cauchy integral formula imply that there exists $C>0$ such that $|f'(w)-f'(z_1)| \le Cd|f'(z_1)|$ for all $w \in B(z_1,d^2)$. Hence
\[ \begin{split}
|(f(z_1)-f(z_2))-f'(z_1)(z_1-z_2)| &\le \int_{z_1}^{z_2} |f'(w)-f'(z_1)| \, |dw| \\
&\le Cd|f'(z_1)(z_1-z_2)| 
\end{split} \]
and consequently
\[ |f(z_1)-f(z_2)| \ge (1-Cd) |f'(z_1)(z_1-z_2)| . \]
Then, for $h \le \delta \defeq d^2$,
\[ \begin{split}
|(f(z_1+ih)-f(z_2+ih))-(f(z_1)-f(z_2))| &\le \int_{z_1}^{z_2} |f'(w+ih)-f'(w)| \, |dw| \\
&\le \int_{z_1}^{z_2} Cd|f'(w)| \, |dw| \\
&\le Cd|f'(z_1)(z_1-z_2)| \\
&\le \frac{Cd}{1-Cd}|f(z_1)-f(z_2)| 
\end{split} \]
and consequently
\[ |f(z_1+ih)-f(z_2+ih)| \ge \left( 1-\frac{Cd}{1-Cd} \right) |f(z_1)-f(z_2)| . \]

Choosing $d$ small enough such that $\frac{Cd}{1-Cd} \le \varepsilon$ implies the claim.
\end{proof}

We choose the sequence $(t_n)$ as in \cref{thm:countable_dense_times}. This implies that $\gamma^0(s_n) \notin \gamma^0([0,s]) \cup \RR$ for all $s < s_n$. Inductively, the same is true for all $\gamma^m$. Moreover, we see that $\gamma^m(I_n) \cap (\gamma^m([0,s]) \cup \RR) = \varnothing$ for all $s < s_n$ and $m \in \NN$.

Note we can choose the sequence $(h_n)$ inductively, where the choice of $h_n$ can depend on $\gamma^0$,...,$\gamma^{n-1}$. This is because although it looks like $\gamma^n$ depend also on future $h_m$ where $m > n$, they actually do not since we have set $\gamma^n$ constant on each $I_m$ for $m > n$.

Let $n \in \NN$. Since $\hat f^{n-1}_{s_n}$ is continuous in $\barH$, the difference $\|\gamma^n-\gamma^{n-1}\|_\infty$ becomes arbitrarily small when $h_n$ is small. The first condition is then immediately satisfied. For the second condition note that $d_{n,n} > 0$ holds automatically when $h_n > 0$. Then it remains to make sure that $d_{k,n} > d_{k,k}/2$ for all $k < n$. But for each $k < n$, we already have $d_{k,n-1} > d_{k,k}/2$ by induction hypothesis. By continuity of the distance function, this holds also for $d_{k,n}$ when $\|\gamma^n-\gamma^{n-1}\|_\infty$ is small enough.

For the third condition, consider any $k < n$. If $s_k < s_n$, there is nothing to do since $\gamma^n = \gamma^{n-1}$ on $[0,s_n]$. In case $s_k > s_n$, we can apply \cref{thm:small_distance_perturbation} with the map $\hat f^{n-1}_{s_n}$ and $A = \gamma^{n-1}_{s_n}(I_k)$ if we know that $\gamma^{n-1}_{s_n}(I_k) \cap \RR = \varnothing$. But this is equivalent to $\gamma^{n-1}(I_k) \cap (\gamma^{n-1}([0,s_n]) \cup \RR) = \varnothing$ which is true by our construction. Therefore, \cref{thm:small_distance_perturbation} implies that $h_n$ can be chosen small enough such that the third condition is preserved from $n-1$ to $n$.

\section{More on trace approximations}

In this section we are going to prove \cref{pr:zero_boundary_time} and \cref{thm:trace_to_driver_continuous}.

\subsection{Proof of \cref{pr:zero_boundary_time}}
\label{se:zero_boundary_time}

We first gather a few general facts.

For a compact set $A \subseteq \barH$ (not necessarily a hull), we can define $\hcap(A) \defeq \lim_{y \to \infty} y \ex[ \Im B^{iy}_{\tau_{\HH \setminus A}} ]$ where $B^{iy}$ denotes Brownian motion started at $iy$ and $\tau_{\HH \setminus A}$ denotes the exit time of $B^{iy}$ from $\HH \setminus A$.

If $A \subseteq B \subseteq \barH$ and $A$ is a compact $\HH$-hull with mapping-out function $g_A$, then $\hcap(B) = \hcap(A)+\hcap(g_A(B\setminus A))$. This can be easily shown from \cite[Proposition 3.41 (3.5)]{Law05} and the strong Markov property of Brownian motion. In particular, with \cite[Proposition 3.42]{Law05} we see that $\hcap(B\setminus A) \ge \hcap(B)-\hcap(A) = \hcap(g_A(B\setminus A))$.

\begin{lemma}\label{le:hcap_sum_bound}
Let $A_1 \subseteq A_2 \subseteq ... \subseteq A_n \subseteq \barH$ be compact $\HH$-hulls. Then
\[ \hcap(A_1 \cup (A_3 \setminus A_2) \cup (A_5 \setminus A_4) \cup ...) \ge \hcap(A_1)+\hcap(A_{2,3})+\hcap(A_{4,5})+... \]
where $A_{i,j} \defeq g_{A_i}(A_j \setminus A_i)$.
\end{lemma}

\begin{proof}
By the above observations, we have
\[ \begin{split}
&\hcap(A_1 \cup (A_3 \setminus A_2) \cup (A_5 \setminus A_4) \cup ...) \\
&\quad = \hcap(A_1)+\hcap(g_1(A_3 \setminus A_2) \cup g_1(A_5 \setminus A_4) \cup ...)\\
&\quad \ge \hcap(A_1)+\hcap(A_{2,3} \cup g_2(A_5 \setminus A_4) \cup ...)
\end{split} \]
and proceed inductively.
\end{proof}

Now we perform the proof of \cref{pr:zero_boundary_time}. It suffices to consider a trace on a compact time interval, say $\gamma\colon [0,1] \to \barH$.  By \cref{th:trace_characterisation} we can find simple traces $\gamma^n$ such that $\gamma^n \to \gamma$ uniformly. By \cref{re:hcap_need} we can assume $\gamma^n$ being parametrised by half-plane capacity.

By the uniform convergence of $\gamma^n$, we can find for any $h > 0$ some $n$ such that $\gamma^{-1}(\RR) \subseteq (\gamma^n)^{-1}(\RR \times [0,h[)$. We would like to show that the latter set has small measure.

The set $(\gamma^n)^{-1}(\RR \times [0,h[)$ consists of a countable number of disjoint intervals $]s_i,t_i[$. Since $\gamma^n$ is simple and parametrised by half-plane capacity, we have $K^n_{t_i} \setminus K^n_{s_i} = \gamma^n(]s_i,t_i])$ and $2\abs{t_i-s_i} = \hcap(K^n_{t_i})-\hcap(K^n_{s_i}) = \hcap(g^n_{s_i}(K^n_{t_i} \setminus K^n_{s_i}))$.

By \cref{le:hcap_sum_bound}, we have for any $I \in \NN$ that
\[ \begin{split}
\sum_{i=1}^I \hcap(g^n_{s_i}(K^n_{t_i} \setminus K^n_{s_i})) 
&\le \hcap\left( \bigcup_{i=1}^I (K^n_{t_i} \setminus K^n_{s_i}) \right) \\
&=  \hcap\left( \bigcup_{i=1}^I \gamma^n(]s_i,t_i]) \right) \\
&\le ch 
\end{split} \]
where $c < \infty$ depends on $\diam\gamma \approx \diam\gamma^n$. Hence, denoting Lebesgue measure by $\abs{\cdot}$,
\[ \abs{\gamma^{-1}(\RR)} \le \abs{(\gamma^n)^{-1}(\RR \times [0,h[)} = \sum_{i \in \NN} \abs{t_i-s_i} \le ch . \]
Since $h > 0$ was arbitrary, this implies $\abs{\gamma^{-1}(\RR)} = 0$.

\subsection{Proof of \cref{thm:trace_to_driver_continuous}}
\label{se:trace_to_driver_continuous}

Since this part is about local convergence, we can restrict ourselves to a compact time interval, say $[0,1]$.

Let $\gamma^n \in C([0,1];\barH)$ be a sequence of chordal Loewner traces, and suppose that $\gamma^n \to \gamma$ uniformly. %\textcolor{red}{If helpful, can assume $\gamma^n$ to be simple.}
Note that such a sequence is equicontinuous, and denote their modulus of continuity by $\omega$, i.e. $\abs{\gamma^n(t)-\gamma^n(s)} \le \omega(\abs{t-s})$ for all $n$, and the same for $\gamma$. As usual, we denote the corresponding hulls by $K_t \defeq \fill(\gamma[0,t])$. Moreover, let $R \defeq \sup_t \diam \gamma_t < \infty$, where $\gamma_t(s) = g_t(\gamma(s))$ as before.

Given $\varepsilon > 0$, we would like to find $\delta > 0$ such that $\norm{\xi-\xi^n}$ is small whenever $\norm{\gamma-\gamma^n} < \delta$.

Let $h_\varepsilon > 0$ such that $\omega(h_\varepsilon) < \varepsilon$. Let $t \in [0,1]$. We follow the proof of \cite[Theorem 4.3]{LMR10} and estimate the difference via
\begin{equation} \begin{split}\label{eq:driver_diff}
\abs{\xi(t)-\xi^n(t)} \le{}& \abs{\xi(t)-g_t(\gamma(t+h))}+\abs{g_t(\gamma(t+h))-g^n_t(\gamma(t+h))}\\
&\quad +\abs{g^n_t(\gamma(t+h))-\xi^n(t)} 
\end{split} \end{equation}
with a suitable $h \in {]0,h_\varepsilon]}$ that we will choose below.

By the half-plane capacity parametrisation and \cite[Lemma 3.4]{JVL11}, we have
\[ \begin{split}
2h_\varepsilon = \hcap(\gamma_t[t,t+h_\varepsilon]) 
&\le c\diam(\gamma_t[t,t+h_\varepsilon])\operatorname{height}(\gamma_t[t,t+h_\varepsilon]) \\
&\le cR\operatorname{height}(\gamma_t[t,t+h_\varepsilon]) .
\end{split} \]
Therefore there exists some $h \in {]0,h_\varepsilon]}$ such that $\Im \gamma_t(t+h) \ge \frac{2h_\varepsilon}{cR}$. By \cite[Lemma 4.5]{LMR10}, it follows that $\dist(\gamma(t+h),K_t) \ge \frac{2h_\varepsilon}{c^2 R} \wedge \frac{4h_\varepsilon^2}{c^4 R^3} \eqdef d$.

By the uniform continuity of $\gamma$, we have $\diam(\gamma[t,t+h]) \le \omega(h) < \varepsilon$, and by \cite[Lemma 4.5]{LMR10} it follows that $\diam(\gamma_t[t,t+h]) \le c(\varepsilon \vee R^{1/2}\varepsilon^{1/2})$. In particular, we have
\[ \abs{\xi(t)-g_t(\gamma(t+h))} = \abs{\gamma_t(t)-\gamma_t(t+h)} \le c(\varepsilon \vee R^{1/2}\varepsilon^{1/2}) \]
which bounds the first difference in \eqref{eq:driver_diff}.

The third difference in \eqref{eq:driver_diff} can be bounded similarly. When we pick $\delta \le d/2$ so that $\delta < d-\delta < \dist(\gamma(t+h),K^n_t)$, then again by \cite[Lemma 4.5]{LMR10}
\[ \abs{g^n_t(\gamma(t+h))-g^n_t(\gamma^n(t+h))} \le c(\delta \vee R^{1/2}\delta^{1/2}) \]
and
\[ \abs{g^n_t(\gamma^n(t+h))-\xi^n(t)} \le c(\varepsilon \vee R^{1/2}\varepsilon^{1/2}) . \]

To bound the second difference in \eqref{eq:driver_diff}, we use \cite[Lemma 4.8]{LMR10}. Let $B \defeq \fill(K_t \cup K^n_t)$.

Pick $\delta \le \frac{d}{2c_0} \wedge \frac{d^2}{4c_0^2 R}$, i.e. we have $\norm{\gamma-\gamma^n} \le \frac{d}{2c_0} \wedge \frac{d^2}{4c_0^2 R}$, where $c_0$ denotes the constant in \cite[Lemma 4.5]{LMR10}.

We now estimate the hyperbolic distance from $\gamma(t+h)$ to $\infty$ in $(\CC \setminus B)^*$ where $^*$ denotes the reflection through $\RR$. By \cite[Lemma 4.4]{LMR10}, we have $\diam g_t(\partial K_t) \le 4R$. By the choice of $\delta$ and \cite[Lemma 4.5]{LMR10} it follows that $g_t(\partial B) \subseteq [a,a+4R+d] \times [0,\frac{d}{2}]$ for some $a \in \RR$.

Denoting by $g_t^*$ the Schwarz reflection of $g_t$ through $\RR$, we have that
\[ \begin{split}
\rho_{(\CC \setminus B)^*}(\gamma(t+h),\infty) 
&= \rho_{g_t^*((\CC \setminus B)^*)}(\gamma_t(t+h),\infty) \\
&\le \rho_{\CC \setminus ([a,a+4R+d] \times [-\frac{d}{2},\frac{d}{2}])}(\gamma_t(t+h),\infty) . \end{split} \]
Recalling that $\Im \gamma_t(t+h) \ge d$, an explicit computation (see the lemma below) shows that the hyperbolic distance is at most $\rho \le \sinh^{-1}(\frac{8R+4d}{d/2}) \le \log(17+\frac{32R}{d})$.

By \cite[Lemma 4.8]{LMR10}, we then have
\[ \begin{split}
&\abs{g_t(\gamma(t+h))-g^n_t(\gamma(t+h))} \\
&\quad \le \abs{g_t(\gamma(t+h))-g_B(\gamma(t+h))}+\abs{g_B(\gamma(t+h))-g^n_t(\gamma(t+h))} \\
&\quad \le 2cR^{1/2} \rho \delta^{1/2} \\
&\quad \le cd\log(17+\frac{32R}{d}) .
\end{split} \]
Since $d$ can be chosen as small as we want, this bounds the second difference in \eqref{eq:driver_diff} and finishes the proof of \cref{thm:trace_to_driver_continuous}.

\begin{lemma}
For $z=x+iy$ with $y > b \ge 0$ we have
\[ \rho_{\hatC \setminus ([-a,a]\times[-b,b])}(z,\infty) \le \sinh^{-1}\left(\frac{4(a+b)}{y-b}\right) . \]
\end{lemma}

\begin{proof}
Let $f\colon \HH \to \HH \setminus ([-a,a]\times[0,b])$ be the hydrodynamically normalised conformal map. By the Schwarz-Christoffel formula, we have
\[ f'(z) = (z-a_1)^{-1/2}(z-a_2)^{1/2}(z-a_3)^{1/2}(z-a_4)^{-1/2} \]
where $a_1,...,a_4$ are the preimages of the points $-a,-a+ib,a+ib,a$. (The multiplicative constant in the formula is determined by $\lim_{z \to \infty} f'(z) = 1$.)

It follows that $\Im f(z) \le b+\Im z$ since $\abs{f'(iy)} \le 1$ and $\Im f'(z) \le 0$ for $\Re z \ge 0$ and $\Im f'(z) \ge 0$ for $\Re z \le 0$.

Call $g$ the Schwarz reflection of $f^{-1}$, so that
\[ \rho_{\hatC \setminus ([-a,a]\times[-b,b])}(z,\infty) = \rho_{\hatC \setminus I}(g(z),\infty) \]
where $I = g(\partial([-a,a]\times[-b,b])) \subseteq \RR$. By \cite[Proposition 4.4]{LMR10}, we have $\diam I \le 4\diam([-a,a]\times[-b,b]) \le 8(a+b)$.

By an explicit computation with the map $h\colon \DD \to \hatC \setminus I$, $h(z) = c(z+\frac{1}{z})$, we get
\[ \rho_{\hatC \setminus I}(g(z),\infty) \le \sinh^{-1}\left(\frac{2c}{\Im g(z)}\right) \le \sinh^{-1}\left(\frac{4(a+b)}{(\Im z)-b}\right) . \]
\end{proof}

%---------------------------------------------------
\bibliographystyle{alpha}
%\bibliography{refs.bib}

\end{document}